\documentclass[reqno,A4paper]{amsart}

\usepackage{amsmath}
\usepackage{amssymb}
\usepackage{amsthm}
\usepackage{enumerate}
\usepackage{mathrsfs} 
\usepackage{eqlist}
\usepackage{array}
\usepackage{hyperref}

\theoremstyle{plain}
\newtheorem{theorem}{Theorem}[section]

\newtheorem{lemma}{Lemma}[section]

\theoremstyle{definition}

\newtheorem*{acknowledgement}{\textup{Acknowledgement}}

\numberwithin{equation}{section}

\begin{document}

\title[Padovan numbers that are concatenations of two distinct repdigits]%
{Padovan numbers that are concatenations of two distinct repdigits}
\author[M. Ddamulira]%
{Mahadi Ddamulira}

\newcommand{\acr}{\newline\indent}

\address{\llap{}Institute of Analysis and Number Theory \acr Graz University of Technology\acr  Kopernikusgasse 24/II\acr A-8010 Graz \acr AUSTRIA}
\email{mddamulira@tugraz.at; mahadi@aims.edu.gh}


\thanks{}

\subjclass[2010]{Primary 11B39, 11D45; Secondary 11D61, 11J86} 
\keywords{Padovan number, repdigit, linear form in logarithms, reduction method.}

\begin{abstract}
Let $ (P_{n})_{n\ge 0} $ be the sequence of Padovan numbers  defined by $ P_0=0 $, $ P_1 =1=P_2$, and $ P_{n+3}= P_{n+1} +P_n$ for all $ n\ge 0 $. In this paper, we find all Padovan numbers that are concatenations of two distinct repdigits.
\end{abstract}

\maketitle

\section{Introduction}
We consider the sequence $ (P_{n})_{n\ge 0} $ of Padovan numbers defined by
\begin{align*}
P_0=0, \quad P_1= 1, \quad P_2=1, \quad \text{and} \quad  P_{n+3}= P_{n+1}+P_n \quad \text{for all}\quad  n\ge 0.
\end{align*}
This is sequence $ A000931  $ on the On-Line Encyclopedia of Integer Sequences (OEIS) \cite{sloa}.
The first few terms of this sequence are
\begin{align*}
(P_{n})_{n\ge 0} = 0, 1, 1, 1, 2, 2, 3, 4, 5, 7, 9, 12, 16, 21, 28, 37, 49, 65, 86, 114, 151, 200, 265, 351, \ldots.
\end{align*}
A repdigit is a positive integer $ N $ that has only one distinct digit when written in its decimal expansion. That is, $ N $ is of the form
\begin{align}\label{rep1}
N=\overline{\underbrace{d\cdots d}_{\ell \text{ times}}}=d\left(\dfrac{10^{\ell}-1}{9}\right),
\end{align}
for some positive integers $ d, \ell $ with $ 0\le d\le 9 $ and $ \ell\ge 1 $. The sequence of repdigits is sequence $ A010785 $ on the OEIS. Diophantine equations involving repdigits and Padovan numbers have been considered in various papers in the recent years. For example: in \cite{Lomeli}, Garc{\'i}a Lomel{\'i} and Hern{\'a}ndez Hern{\'a}ndez  found all repdigits that can be written as a sum of two Padovan numbers; in \cite{Ddamulira}, the author  found all repdigits that can be written as a sum of three Padovan numbers.

\section{Main Result}
In this paper, we study the problem of finding all Padovan numbers that are concatenations of two repdigits. More precisely, we completely solve the Diophantine equation
\begin{eqnarray}
P_n=\overline{\underbrace{d_1\cdots d_1}_{\ell_1 \text{ times}} \underbrace{d_2\cdots d_2}_{\ell_2 \text{ times}}}=d_1\left(\dfrac{10^{\ell_1}-1}{9}\right)\times 10^{\ell_2}+d_2\left(\dfrac{10^{\ell_2}-1}{9}\right),\label{Problem}
\end{eqnarray}
in non-negative integers $ (n, d_1, d_2, \ell_1, \ell_2) $ with $ n \ge 0 $, $ d_1 \neq d_2\in\{0, \ldots, 9\} $, $ d_1>0 $, and $ \ell_1, \ell_2 \ge 1 $.  We discard the case $ d_1=d_2 $ since it was already considered by Garc{\'i}a Lomel{\'i} and Hern{\'a}ndez Hern{\'a}ndez in \cite{Lomeli}, and by the result of the author in \cite{Ddamulira}.

Our main result is the following.
\begin{theorem}\label{Main}
The only Padovan numbers that are concatenations of two distinct repdigits are $$ P_n \in \{12, 16,21,28, 37, 49, 65, 86, 114, 200 \}.$$
\end{theorem}
This paper is inspired by the results of Alahmadi, Altassan, Luca, and Shoaib \cite{Luca}, in which they find all Fibonacci numbers that are concatenations of two repdigits, and Rayaguru and Panda \cite{Panda}, in which they find all balancing numbers that are concatenations of two repdigits. Our method of proof involves the application of Baker's theory for linear forms in logarithms of algebraic numbers, and the Baker-Davenport reduction procedure. Computations are done with the help of a computer program in Mathematica.
\section{Preliminary results}
\subsection{The Padovan sequence} 
Here, we recall some important properties of the Padovan sequence $ \{P_n\}_{n\geq 0} $. The characteristic equation
\begin{align*}
 \Psi(x):=x^3-x-1 = 0,
\end{align*}
has roots $ \alpha, \beta, \gamma = \bar{\beta} $, where
\begin{eqnarray}\label{Pado1}
\alpha =\dfrac{r_1+r_2}{6}, \qquad \beta = \dfrac{-(r_1+r_2)+\sqrt{-3}(r_1-r_2)}{12}
\end{eqnarray}
and
\begin{eqnarray}\label{Pado2}
r_1=\sqrt[3]{108+12\sqrt{69}} \quad \text{and}\quad r_2=\sqrt[3]{108-12\sqrt{69}}.
\end{eqnarray}
Furthermore, the Binet formula is given by
\begin{eqnarray}\label{Pado3}
P_n = a\alpha^{n}+b\beta^{n}+c\gamma^{n} \qquad \text{ for all} \quad n\ge 0,
\end{eqnarray}
where
\begin{eqnarray}\label{Pado4}
\quad a=\dfrac{\alpha+1}{(\alpha-\beta)(\alpha-\gamma)}, \quad b= \dfrac{\beta +1}{(\beta-\alpha)(\beta-\gamma)}, \quad c = \dfrac{\gamma+1}{(\gamma-\alpha)(\gamma-\beta)}=\bar{b}.
\end{eqnarray}
The minimal polynomial of $ a $ over the integers is given by
\begin{align*}
23x^3-5x-1,
\end{align*}
has zeros $ a, ~b, ~c $ with $ |a|, ~|b|, ~|c| < 1 $.
Numerically, the following estimates hold:
\begin{equation}\label{Pado5}
\begin{aligned}
1.32&<\alpha<1.33;\\
0.86 < |\beta|&=|\gamma|=\alpha^{-\frac{1}{2}}< 0.87;\\
0.54&<a<0.55;\\
0.28<&|b|=|c|<0.29.
\end{aligned}
\end{equation}
From \eqref{Pado1}, \eqref{Pado2} and \eqref{Pado5}, it is easy to see that the contribution the complex conjugate roots $ \beta $ and $ \gamma $, to the right-hand side of \eqref{Pado3}, is very small. In particular, setting
\begin{eqnarray}\label{Pado6}
e(n):=P_n-a\alpha^{n}=b\beta^{n}+c\gamma^{n}\quad \text{ then } \quad |e(n)|< \dfrac{1}{\alpha^{n/2}},
\end{eqnarray}
holds for all $ n\ge 1 $.
Furthermore, by induction, one can prove that 
\begin{eqnarray}\label{Pado7}
\alpha^{n-3}\leq P_n \leq \alpha^{n-1} \quad \text{holds for all }\quad n\geq 1.
\end{eqnarray}
Let $ \mathbb{K}:=\mathbb{Q}(\alpha, \beta) $ be the splitting field of the polynomial $ \Psi $ over $ \mathbb{Q} $. Then, $ [\mathbb{K}, \mathbb{Q}]=6 $. Furthermore, $ [\mathbb{Q}(\alpha):\mathbb{Q}]=3 $. The Galois group of $ \mathbb{K} $ over $ \mathbb{Q} $ is given by
\begin{eqnarray*}
\mathcal{G}:=\text{Gal}(\mathbb{K/Q})\cong \{(1), (\alpha\beta),(\alpha\gamma), (\beta\gamma), (\alpha\beta\gamma), (\alpha\gamma\beta)\} \cong S_3.
\end{eqnarray*}
Thus, we identify the automorphisms of $ \mathcal{G} $ with the permutations of the zeros of the polynomial $ \Psi $. For example, the permutation $ (\alpha\beta) $ corresponds to the automorphism $ \sigma: \alpha \to \beta, ~\beta \to \alpha, ~\gamma \to \gamma $.
\subsection{Linear forms in logarithms}
Let $ \eta $ be an algebraic number of degree $ d $ with minimal primitive polynomial over the integers
$$ a_{0}x^{d}+ a_{1}x^{d-1}+\cdots+a_{d} = a_{0}\prod_{i=1}^{d}(x-\eta^{(i)}),$$
where the leading coefficient $ a_{0} $ is positive and the $ \eta^{(i)} $'s are the conjugates of $ \eta $. Then the \textit{logarithmic height} of $ \eta $ is given by
\begin{align*}
h(\eta) := \dfrac{1}{d}\left( \log a_{0} + \sum_{i=1}^{d}\log\left(\max\{|\eta^{(i)}|, 1\}\right)\right).
\end{align*}
In particular, if $ \eta = p/q $ is a rational number with $ \gcd (p,q) = 1 $ and $ q>0 $, then $ h(\eta) = \log\max\{|p|, q\} $. The following are some of the properties of the logarithmic height function $ h(\cdot) $, which will be used in the next section of this paper without reference:
\begin{align*}
h(\eta_1\pm \eta_2) &\le h(\eta_1) +h(\eta_2) +\log 2;\\
h(\eta_1\eta_2^{\pm 1})&\le h(\eta_1) + h(\eta_2);\\
h(\eta^{s}) &= |s|h(\eta) \quad  (s\in\mathbb{Z}). 
\end{align*}

We recall the result of Bugeaud, Mignotte, and Siksek (\cite{BuMiSi}, Theorem 9.4, pp. 989), which is a  modified version of the result of Matveev \cite{MatveevII}, which is one of our main tools in this paper.
\begin{theorem}\label{Matveev11} Let $\eta_1,\ldots,\eta_t$ be positive real algebraic numbers in a real algebraic number field  $\mathbb{K} \subset \mathbb{R}$ of degree $D$, $b_1,\ldots,b_t$ be nonzero integers, and assume that
\begin{align*}
\Lambda:=\eta_1^{b_1}\cdots\eta_t^{b_t} - 1\neq 0.
\end{align*}
Then,
\begin{align*}
\log |\Lambda| > -1.4\times 30^{t+3}\times t^{4.5}\times D^{2}(1+\log D)(1+\log B)A_1\cdots A_t,
\end{align*}
where
$$
B\geq\max\{|b_1|, \ldots, |b_t|\},
$$
and
$$A
_i \geq \max\{Dh(\eta_i), |\log\eta_i|, 0.16\},\qquad {\text{for all}}\qquad i=1,\ldots,t.
$$
\end{theorem}

\subsection{Reduction procedure}
During the calculations, we get upper bounds on our variables which are too large, thus we need to reduce them. To do so, we use some result from the theory of continued fractions. For a nonhomogeneous linear form in two integer variables, we use a slight variation of a result due to Dujella and Peth{\H o} (\cite{dujella98}, Lemma 5a). For a real number $X$, we write  $\|X\|:= \min\{|X-n|: n\in\mathbb{Z}\}$ for the distance from $X$ to the nearest integer.
\begin{lemma}\label{Dujjella}
Let $M$ be a positive integer, $\frac{p}{q}$ be a convergent of the continued fraction expansion of the irrational number $\tau$ such that $q>6M$, and  $A,B,\mu$ be some real numbers with $A>0$ and $B>1$. Furthermore, let $\varepsilon: = \|\mu q\|-M\|\tau q\|$. If $ \varepsilon > 0 $, then there is no solution to the inequality
$$
0<|u\tau-v+\mu|<AB^{-w},
$$
in positive integers $u,v$, and $w$ with
$$ 
u\le M \quad {\text{and}}\quad w\ge \dfrac{\log(Aq/\varepsilon)}{\log B}.
$$
\end{lemma}

The following Lemma is also useful. It is due to G\'uzman S\'anchez and Luca (\cite{guzmanluca}, Lemma 7). 
\begin{lemma}
\label{gl}
If $r\ge 1$, $H>(4r^2)^r$,  and $H>L/(\log L)^r$, then
$$
L<2^rH(\log H)^r.
$$
\end{lemma}

\section{The proof of Theorem \ref{Main}}
\subsection{The small ranges}
With the help of Mathematica, we checked all the solutions to the Diophantine equation \eqref{Problem} in the ranges $ d_1\neq  d_2\in\{0, \ldots,  9 \}$, $ d_1>0 $, and $ 1\le \ell_1, \ell_2 \le n \le 500  $ and found only the solutions stated in Theorem \ref{Main}. From now on we assume that $ n>500 $.
\subsection{The initial bound on $ n $}
We rewrite \eqref{Problem} as
\begin{align}\label{kala1}
P_n=\dfrac{1}{9}\left(d_1\times 10^{\ell_1+\ell_2}-(d_1-d_2)\times 10^{\ell_2} - d_2\right).
\end{align}
We prove the following lemma, which gives a relation on the size of $ n $ versus $ \ell_1+\ell_2 $.
\begin{lemma}\label{lemx}
All solutions of the Diophantine equation \eqref{kala1} satisfy
\begin{align*}
(\ell_1+\ell_2)\log 10 -3 < n\log\alpha < (\ell_1+\ell_2)\log 10 +1. 
\end{align*}
\end{lemma}
\begin{proof}
The proof follows easily from \eqref{Pado7}. One can see from \eqref{kala1} that
\begin{align*}
\alpha^{n-3}\le P_n < 10^{\ell_1+\ell_2}.
\end{align*}
Taking the logarithm on both sides, we get that 
\begin{align*}
(n-3)\log \alpha < (\ell_1+\ell_2)\log 10,
\end{align*}
which leads to
\begin{align}\label{lemxx1}
n\log\alpha < (\ell_1+\ell_2)\log 10+3 \log\alpha < (\ell_1+\ell_2)\log 10 +1.
\end{align}
For the lower bound, we have from \eqref{kala1} that
\begin{align*}
10^{\ell_1+\ell_2-1}<P_n \le  \alpha^{n-1}.
\end{align*}
Taking the logarithm on both sides, we get that
\begin{align*}
(\ell_1+\ell_2-1)\log 10 < (n-1)\log\alpha,
\end{align*}
which leads to
\begin{align}\label{lemxx2}
(\ell_1+\ell_2)\log 10-3 < (\ell_1+\ell_2-1)\log 10+\log\alpha < n\log \alpha.
\end{align}
Comparing \eqref{lemxx1} and \eqref{lemxx2} gives the result in the lemma.
\end{proof}

Next, we examine \eqref{kala1} in two different steps.
\paragraph{\bf Step 1} Substituting \eqref{Pado3} in \eqref{kala1}, we get that
\begin{align*}
a\alpha^{n}+b\beta^{n}+c\gamma^{n}=\dfrac{1}{9}\left(d_1\times 10^{\ell_1+\ell_2}-(d_1-d_2)\times 10^{\ell_2} - d_2\right).
\end{align*}
By \eqref{Pado6}, this is equivalent to
\begin{align*}
9a\alpha^{n}-d_1\times 10^{\ell_1+\ell_2}=-9e(n)-(d_1-d_2)\times 10^{\ell_2}-d_2,
\end{align*}
from which we deduce that
\begin{align*}
\left|9a\alpha^{n}-d_1\times 10^{\ell_1+\ell_2}\right|&=\left|9e(n)+(d_1-d_2)\times 10^{\ell_2}+d_2\right|\\
&\le 9\alpha^{-n/2}+9\times 10^{\ell_2}+9\\
&< 30\times 10^{\ell_2}.
\end{align*}
Thus, dividing both sides by $ d_1\times 10^{\ell_1+\ell_2} $ we get that
\begin{align}\label{bad1}
\left|\left(\dfrac{9a}{d_1}\right)\cdot \alpha^{n}\cdot 10^{-\ell_1-\ell_2}   -1\right|< \dfrac{30\times 10^{\ell_2}}{d_1\cdot 10^{\ell_1+\ell_2}}\le \dfrac{30}{10^{\ell_1}}.
\end{align}
Put
\begin{align}\label{xx1}
\Lambda_1:=\left(\dfrac{9a}{d_1}\right)\cdot \alpha^{n}\cdot 10^{-\ell_1-\ell_2}-1.
\end{align}
Next, we apply Theorem \ref{Matveev11} on \eqref{xx1}. First, we need to check that $ \Lambda_1 \neq 0$. If it were, then we would get that
\begin{align*}
a\alpha^{n}=\dfrac{d_1}{9}\cdot 10^{\ell_1+\ell_2}.
\end{align*}
Now, we apply the automorphism $ \sigma $ of the Galois group $ \mathcal{G} $ on both sides and take absolute values as follows.
\begin{align*}
\left|\dfrac{d_1}{9}\cdot 10^{\ell_1+\ell_2}\right|=\left|\sigma (a\alpha^{n})\right| = \left|b\beta^{n}\right|<1,
\end{align*}
which is false. Thus, $ \Lambda_1\neq 0 $. So, we apply Theorem \ref{Matveev11} on \eqref{xx1} with the data:
\begin{align*}
t:=3, \quad \eta_1:=\dfrac{9a}{d_1}, \quad \eta_2:=\alpha, \quad  \eta_3:=10, \quad b_1:=1, \quad b_2:=n, \quad b_3:=-\ell_1-\ell_2.
\end{align*}
By Lemma \ref{lemx}, we have that $ \ell_1+\ell_2 < n $. Therefore, we can take $ B:=n $. Observe that $ \mathbb{K}:=\mathbb{Q}(\eta_1,\eta_2, \eta_3)=\mathbb{Q}(\alpha)$, since $ a=\alpha (\alpha + 1)/(2\alpha+3) $, so $ D:=3 $. We have
\begin{align*}
h(\eta_1)=h(9a/d_1)\le h(9)+h(a)+h(d_1)\le \log 9+\frac{1}{3}\log 23 +\log 9 \le 5.44.
\end{align*}
Furthermore, $ h(\eta_2)=h(\alpha)=(1/3)\log \alpha $ and $ h(\eta_3)=h(10)=\log 10 $. Thus, we can take
\begin{align*}
A_1:=16.32, \quad A_2:=\log\alpha, \quad \text{and} \quad A_3:=3\log 10.
\end{align*}
Theorem \ref{Matveev11} tells us that
\begin{align*}
\log|\Lambda_1|&>-1.4\times 30^{6}\times 3^{4.5}\times 3^{2}(1+\log 3)(1+\log n)(16.32)(\log\alpha)(3\log 10)\\
&>-1.45\times 10^{30}(1+\log n).
\end{align*}
Comparing the above inequality with \eqref{bad1} gives
\begin{align*}
\ell_1\log 10-\log 30 < 1.45\times 10^{30}(1+\log n),
\end{align*}
leading to
\begin{align}\label{bad2}
\ell_1\log 10 < 1.46\times 10^{30}(1+\log n).
\end{align}

\paragraph{\bf Step 2}
By \eqref{Pado6}, we rewrite \eqref{kala1} as
\begin{align*}
9a\alpha^{n}-\left(d_1\times 10^{\ell_1}-(d_1-d_2)\right)\times 10^{\ell_2}=-9e(n)-d_2,
\end{align*}
from which we deduce that
\begin{align*}
\left|9a\alpha^{n}-\left(d_1\times 10^{\ell_1}-(d_1-d_2)\right)\times 10^{\ell_2}\right|=\left|9e(n)+d_2\right|
\le 9\alpha^{-n/2}+9
< 18.
\end{align*}
Thus, dividing both sides by $ 9a\alpha^{n} $ we get that
\begin{align}\label{bad3}
\left|\left(\dfrac{d_1\times 10^{\ell_1}-(d_1-d_2)}{9a}\right)\cdot \alpha^{-n}\cdot 10^{\ell_2}   -1\right|< \dfrac{18}{9a\alpha^{n}}< \dfrac{4}{\alpha^{n}}.
\end{align}
Put
\begin{align}\label{xx12}
\Lambda_2:=\left(\dfrac{d_1\times 10^{\ell_1}-(d_1-d_2)}{9a}\right)\cdot \alpha^{-n}\cdot 10^{\ell_2}   -1.
\end{align}
Next, we apply Theorem \ref{Matveev11} on \eqref{xx12}. First, we need to chech that $ \Lambda_2 \neq 0$. If not, then we would get that
\begin{align*}
a\alpha^{n}=\left(\dfrac{d_1\times 10^{\ell_1}-(d_1-d_2)}{9}\right)\cdot 10^{\ell_2}.
\end{align*}
Then, we apply the automorphism $ \sigma $ of the Galois group $ \mathcal{G} $ on both sides and take absolute values as follows.
\begin{align*}
\left|\left(\dfrac{d_1\times 10^{\ell_1}-(d_1-d_2)}{9}\right)\cdot 10^{\ell_2}\right|=\left|\sigma (a\alpha^{n})\right| = \left|b\beta^{n}\right|<1,
\end{align*}
which is false. Thus, $ \Lambda_2\neq 0 $. So, we apply Theorem \ref{Matveev11} on \eqref{xx12} with the data:
\begin{align*}
t:=3, \quad \eta_1:=\dfrac{d_1\times 10^{\ell_1}-(d_1-d_2)}{9a}, \quad \eta_2:=\alpha, \quad  \eta_3:=10, \quad b_1:=1, \quad b_2:=-n, \quad b_3:=\ell_2.
\end{align*}
As before, we have that $ \ell_2<n $. Thus, we can take $ B:=n $. Similary, $ \mathbb{Q}(\eta_1, \eta_2, \eta_3)=\mathbb{Q}(\alpha) $, so we take $ D:=3 $. Furthermore, we have

\begin{align*}
h(\eta_1)&=h\left(\dfrac{d_1\times 10^{\ell_1}-(d_1-d_2)}{9a}\right)\\
&\le h(d_1\times 10^{\ell_1}-(d_1-d_2)) + h(9a)\\
&\le h(d_1\times 10^{\ell_1})+h(d_1-d_2)+h(9)+h(a)+\log 2\\
&\le h(d_1)+\ell_1  h(10)+ h(d_1)+h(d_2)+h(9)+h(a)+2\log 2\\
&\le \ell_1\log 10+4\log 9+\dfrac{1}{3}\log 23+2\log 2\\
&\le 1.46\times 10^{30}(1+\log n)+4\log 9+\dfrac{1}{3}\log 23+2\log 2 \quad (\text{by } \eqref{bad2})\\
&<1.48\times 10^{30}(1+\log n).
\end{align*}
We also consider
\begin{align*}
|\log (\eta_1)|&=\left|\log \left(\dfrac{d_1\times 10^{\ell_1}-(d_1-d_2)}{9a}\right)\right|\\
&\le |\log (d_1\times 10^{\ell_1}-(d_1-d_2))| + \log (9a)\\
&\le \log (d_1\times 10^{\ell_1})+\left|\log \left(1-\frac{d_1-d_2}{d_1\times 10^{\ell_1}} \right)\right|+\log 9+\log a\\
&\le \ell_1\log 10+ \log d_1+\log 9+\log a+ \dfrac{|d_1-d_2|}{d_1\times 10^{\ell_1}}+ \dfrac{1}{2}\left(\dfrac{|d_1-d_2|}{d_1\times 10^{\ell_1}}\right)^{2}+\cdots\\
&\le \ell_1\log 10+2\log 9+\log 1.33+\dfrac{1}{10^{\ell_1}}+ \dfrac{1}{2\times 10^{2\ell_1}}+\cdots\\
&\le 1.46\times 10^{30}(1+\log n)+2\log 9+ \dfrac{1}{10^{\ell_1}-1} \quad (\text{by } \eqref{bad2})\\
&<1.48\times 10^{30}(1+\log n).
\end{align*}
So, $ Dh(\eta_1)> |\log \eta_1| $. Thus, we can take
\begin{align*}
A_1:=4.44\times 10^{30}(1+\log n), \quad A_2:=\log\alpha, \quad \text{and} \quad A_3=3\log 10.
\end{align*}
Theorem \ref{Matveev11} tells us that
\begin{align*}
\log|\Lambda_2|&>-1.4\times 30^{6}\times 3^{4.5}\times 3^2(1+\log 3)(1+\log n)(4.44\times 10^{30}(1+\log n))(\log\alpha)(3\log 10)\\
&>-2.38\times 10^{43}(1+\log n)^2.
\end{align*}
Comparing the above inequality with \eqref{bad3} gives,
\begin{align*}
n\log \alpha - \log 4 < 2.38\times 10^{43}(1+\log n)^2,
\end{align*}
which is equivalent to
\begin{align}\label{bad4}
n < 1.70 \times 10^{44}(\log n)^2.
\end{align}
Applying Lemma \ref{gl} on \eqref{bad4} with the data $ r=2 $, $H:=1.70 \times 10^{44}$, and $ L:=n $, gives
\begin{align*}
n<7.38\times 10^{48}.
\end{align*}
Lemma \ref{lemx} implies that
\begin{align*}
\ell_1+\ell_2 < 9.15\times 10^{47}.
\end{align*}
We have just proved the following lemma.
\begin{lemma}\label{lemxxx}
All solutions to the Diophantine equation \eqref{kala1} satisfy
\begin{align*}
\ell_1+\ell_2 < 9.15\times 10^{47} \quad \text{and} \quad n< 7.38\times 10^{48}.
\end{align*}
\end{lemma}
\subsection{Reducing the bounds} The bounds given in Lemma \ref{lemxxx} are to large to carry out meaningful computation. Thus, we need to reduce them. To do so, we apply Lemma \ref{Dujjella} as follows.

First, we return to \eqref{bad1} and put
\begin{align*}
\Gamma_1:=(\ell_1+\ell_2)\log 10-n\log\alpha - \log \left(\frac{9a}{d_1}\right).
\end{align*}
The inequality \eqref{bad1} can be rewritten as
\begin{align*}
\left|e^{-\Gamma_1}-1\right| < \dfrac{30}{10^{\ell_1}}.
\end{align*}
Assume that $ \ell_1 \ge 2 $, then the right--hand side in the above inequality is at most $ 3/10 <1/2 $. The inequality $ |e^x-1|<y $ for real values of $ x $ and $ y $ implies that $ x<2y $. Thus,
\begin{align*}
|\Gamma_1| < \dfrac{60}{10^{\ell_1}},
\end{align*}
which implies that 
\begin{align*}
\left|(\ell_1+\ell_2)\log 10-n\log\alpha - \log \left(\frac{9a}{d_1}\right)\right|< \dfrac{60}{10^{\ell_1}}.
\end{align*}
Dividing through by $ \log\alpha $ gives
\begin{align*}
\left|(\ell_1+\ell_2)\frac{\log 10}{\log\alpha}-n+ \left(\frac{\log ({d_1}/{9a})}{\log\alpha}\right)\right|< \dfrac{60}{10^{\ell_1}\log\alpha}.
\end{align*}
So, we apply Lemma \ref{Dujjella} with the data:
\begin{align*}
\tau:=\frac{\log 10}{\log\alpha}, \quad \mu(d_1):=\frac{\log ({d_1}/{9a})}{\log\alpha}, \quad A:=\dfrac{60}{\log\alpha}, \quad B:=10, \quad 1\le d_1\le 9.
\end{align*}
Let $ \tau = [a_{0}; a_{1}, a_{2}, \ldots]=[8; 5, 3, 3, 1, 5, 1, 8, 4, 6, 1, 4, 1, 1, 1, 9, 1, 4, 4, 9, 1, 5, 1, 1, 1, 5, 1, 1, 1, 2, 1, \ldots] $ be the continued fraction expansion of $ \tau $. We choose $M:=8\times 10^{48}$ which is the upper bound on $ \ell_1+\ell_2 $. With the help of Mathematica, we find out that the convergent 
\begin{align*}
\dfrac{p}{q}= \dfrac{p_{106}}{q_{106}} = \dfrac{177652856036642165557187989663314255133456297895465}{21695574963444524513646677911090250505443859600601},
\end{align*}
 is such that $ q=q_{106}>6M $. Furthermore, it yields $ \varepsilon > 0.0375413 $, and therefore
\begin{align*}
\ell_1 \leq \dfrac{\log\left((60/\log \alpha) q/\varepsilon\right)}{\log 10} < 53.
\end{align*}
Thus, we have that  $ \ell_1\leq 53 $. The case $ \ell_1 < 2 $ holds as well since $ \ell_1 < 2 < 53 $.

For fixed $  d_1 \neq  d_2 \in \{0, \ldots, 9\}$, $ d_1>0 $, and $ 1\le \ell_1 \le 53 $, we return to \eqref{bad3} and put
\begin{align*}
\Gamma_2:=\ell_2\log 10 -n\log\alpha + \log \left(\dfrac{d_1\times 10^{\ell_1}-(d_1-d_2)}{9a}\right).
\end{align*}
From the inequality \eqref{bad3}, we have that
\begin{align*}
\left|e^{\Gamma_2}-1\right|<\dfrac{4}{\alpha^{n}}.
\end{align*}
Since $ n>500 $, the right--hand side of the above inequality is less than $ 1/2 $. Thus, the above inequality implies that
\begin{align*}
\left|\Lambda_1\right| < \dfrac{8}{\alpha^{n}},
\end{align*}
which leads to
\begin{align*}
\left|\ell_2\log 10 -n\log\alpha + \log \left(\dfrac{d_1\times 10^{\ell_1}-(d_1-d_2)}{9a}\right)\right|<\dfrac{8}{\alpha^{n}}.
\end{align*}
Dividing through by $ \log\alpha $ gives,
\begin{align*}
\left|\ell_2\left(\frac{\log 10}{\log\alpha}\right) -n + \frac{\log \left((d_1\times 10^{\ell_1}-(d_1-d_2))/9a\right)}{\log\alpha}\right|<\dfrac{8}{\alpha^{n}\log\alpha}.
\end{align*}
Again, we apply Lemma \ref{Dujjella} with the data:
\begin{align*}
\tau:=\dfrac{\log 10}{\log\alpha}, \quad \mu(d_1, d_2):=\frac{\log \left((d_1\times 10^{\ell_1}-(d_1-d_2))/9a\right)}{\log\alpha}, \quad A:=\dfrac{8}{\log\alpha}, \quad B:=\alpha.
\end{align*}
We take the same $ \tau $ and its convergent $ p/q=p_{106}/q_{106} $ as before. We choose $ \ell_2<8\times 10^{48}:=M $.  With the help of Mathematica,  we get that $ \varepsilon > 0.0000903006 $, and therefore 
\begin{align*}
n \leq \dfrac{\log\left((8/\log \alpha) q/\varepsilon\right)}{\log \alpha} < 446.
\end{align*}
Thus, we have that  $ n\leq 446 $, contradicting the working assumption that $ n>500 $. Hence, Theorem \ref{Main} is proved. \qed

\begin{acknowledgement}
The author thanks the anonymous referees for their useful comments and suggestions that greatly improved  the quality of presentation of the current paper. The author is supported by the Austrian Science Fund (FWF) projects: F5510-N26 -- Part of the special research program (SFB), ``Quasi-Monte Carlo Methods: Theory and Applications'' and W1230 --``Doctoral Program Discrete Mathematics''.
\end{acknowledgement}

\def\cprime{$'$}

\end{document}